\documentclass{amsart}
\pdfoutput=1 
\usepackage{mathtools}
\usepackage{amssymb}
\usepackage{hyperref}

\title[Quasi-isometries and the contracting Gromov product boundary]{Quasi-isometries need not induce homeomorphisms of contracting
  boundaries with the Gromov product topology}
\author{Christopher H. Cashen} 
\thanks{Supported by the Austrian Science Fund
(FWF):M1717-N25.}
\address{
Faculty of Mathematics\newline \indent
University of Vienna\newline \indent
Oskar-Morgenstern-Platz 1\newline \indent
1090 Vienna, Austria}
\email{\href{mailto:christopher.cashen@univie.ac.at}{christopher.cashen@univie.ac.at}}
\urladdr{\href{http://www.mat.univie.ac.at/~cashen}{\nolinkurl{http://www.mat.univie.ac.at/~cashen}}}

\hypersetup{
    pdftitle={Quasi-isometries need not induce homeomorphisms of Gromov boundaries},    
    pdfauthor={Christopher H. Cashen},     
    pdfkeywords={Gromov boundary, quasi-isometry, contracting geodesic}, 
    colorlinks=true,       
    linkcolor=black,          
    citecolor=black,        
    filecolor=black,      
    urlcolor=black           
}
\newtheorem*{proposition}{Proposition}
\newcommand{\bdry}{\partial} 
\newcommand{\from}{\colon\thinspace} 
\mathtoolsset{centercolon} 
\DeclareMathOperator{\diam}{\mathrm{diam}}
\begin{document}
\begin{abstract}
We consider a `contracting boundary' of a proper geodesic metric space
consisting of equivalence classes of geodesic rays that behave like
geodesics in a hyperbolic space. 
We  topologize this set via the Gromov product,
in analogy to the topology of the boundary of a hyperbolic space.
We show that when the space is not hyperbolic, quasi-isometries do not
necessarily give homeomorphisms of this boundary.
Continuity can fail even when the spaces are required to be CAT(0).
We show this by constructing an explicit example.
\end{abstract}
\maketitle
\section{Introduction}
In an extremely influential paper, Gromov \cite{Gro87} introduced hyperbolic spaces and their
boundaries. 
Among myriad applications, the topological type of the boundary
provides a quasi-isometry invariant of the space, since quasi-isometries of hyperbolic spaces
extend to homeomorphisms of their boundaries.

Recently Charney and Sultan \cite{ChaSul15} introduced a
quasi-isometry invariant `contracting
boundary' for CAT(0) spaces, consisting of those equivalence classes of
geodesic rays that are `contracting', which is to say that they behave
like geodesic rays in a hyperbolic space in a certain quantifiable
way.
As a set, the contracting boundary of a CAT(0) space can be naturally viewed as a subset of
the visual boundary of the space.
A quasi-isometry does induce a bijection of this contracting subset,
even though it does not necessarily induce a homeomorphism of the
entire visual boundary.
Charney and Sultan were unable to determine if this bijection is a homeomorphism with respect
to the subspace topology. 
Instead, they define a finer topology that they show to be
quasi-isometrically invariant. 
We answer their question in the negative:
quasi-isometries of CAT(0) spaces do not, in general,
induce homeomorphisms of the contracting boundary with the subspace topology.
We do so by constructing an explicit example. 

\section{The contracting boundary and Gromov product topology}
Let $X$ be a proper geodesic metric space.
Let $\gamma$ be a geodesic ray in $X$, and define the closest point
projection map $\pi_\gamma\from
X\to 2^X$ by $\pi_\gamma(x):=\{y\in \gamma\mid
d(x,y)=d(x,\gamma)\}$.
Properness of $X$ guarantees that the empty set is not in the image of $\pi_\gamma$.

A geodesic ray $\gamma$ in $X$ is \emph{contracting} if there exists a
non-decreasing, eventually non-negative function $\rho$ such that
$\lim_{r\to\infty}\rho(r)/r=0$ and such that for
all $x$ and $y$ in $X$, if $d(x,y)\leq d(x,\gamma)$ then $\diam
\pi_\gamma(x)\cup\pi_\gamma(y)\leq \rho(d(x,\gamma))$.
The ray is \emph{strongly contracting} if the function $\rho$ can be
chosen to be bounded.

A geodesic ray $\gamma$ in $X$ is \emph{Morse} if there exists a 
function $\mu$ such that if $\alpha$ is a 
$(\lambda,\epsilon)$--quasi-geodesic with endpoints on $\gamma$, then 
$\alpha$ is contained in the $\mu(\lambda,\epsilon)$--neighborhood of $\gamma$. 

It is not hard to show that a contracting ray is Morse.
Cordes \cite{Cor15} generalizes the Charney-Sultan construction by building a
`Morse boundary' consisting of equivalence classes of Morse geodesic
rays in an arbitrary geodesic metric space. 
In fact, the Morse and contracting properties are equivalent in geodesic metric 
spaces \cite{ArzCasGrub}, so we can just as well call Cordes's
construction the contracting boundary, where we allow rays satisfying the
more general version of contraction defined above.

Let us describe the points of the contracting boundary.
For points $x,\,y,\,z\in X$, the \emph{Gromov product} of $x$ and $y$
with respect to $z$ is defined by:
\[(x\cdot y)_z:=\frac{1}{2}(d(x,z)+d(y,z)-d(x,y))\]

Fix a basepoint $o\in X$ and consider contracting geodesic rays based at $o$.
Define an equivalence relation by $\alpha\sim\beta$ if $\lim_{i,j\to\infty}(\alpha(i)\cdot\beta(j))_o=\infty$.
This relation is transitive on contracting geodesic rays
because contracting rays are Morse and Morse rays are related if and
only if they are at bounded Hausdorff distance from one another.
Define the contracting boundary $\bdry_c X$ to be the set of equivalence classes.
It is easy to see that a quasi-isometry $\phi$ of $X$ induces a
bijection $\bdry_c\phi$ of $\bdry_cX$.
 It remains to define a topology on $\bdry_cX$ and check continuity of
 $\bdry_c\phi$.

The topology is defined by restricting to $\bdry_c X$ the usual
construction of the `ideal' or `Gromov' boundary (cf. \cite{Gro87,BriHae99,
  BucKok09}).
Extend the Gromov product to $\bdry_cX$ by:
\[(\eta\cdot\zeta)_o:=\sup_{\alpha\in\eta,\,\beta\in\zeta}\liminf_{i,j\to\infty}(\alpha(i)\cdot\beta(j))_o\]
Given $\eta\in\bdry_c X$ and $r>0$, define
$U(\eta,r):=\{\zeta\in\bdry_c X\mid (\eta\cdot\zeta)_o\geq r\}$.
Define the \emph{Gromov product topology on $\bdry_cX$} to be the
topology such that a set $U\subset\bdry_c X$ is open if for every $\eta\in U$ there
exists an $r>0$ such that $U(\eta,r)\subset U$.
Denote the
contracting boundary with this topology $\bdry_c^{Gp}X$.

When $X$ is hyperbolic $\bdry_c^{Gp}X$ is the usual Gromov boundary.
When $X$ is CAT(0) $\bdry_c^{Gp}X$ is homeomorphic to the
contracting subset of the visual boundary with the subspace topology.

Note that $\mathcal{U}_\eta:=\{U(\eta,r)\mid r>0\}$ is not necessarily a neighborhood
basis at $\eta$ in this topology. 
We do not need this fact for the conclusions of Section
\ref{sec:path}, but, as it may be of separate interest, we give a
sufficient condition. Some spaces satisfy a \emph{contraction alternative} in the sense that
every  geodesic ray is either strongly contracting or not contracting. 
We will say that such a space \emph{is CA}.
By \cite{ArzCasGrub}, CA is equivalent to ``every Morse geodesic ray is
strongly contracting.''
Examples of CA spaces include hyperbolic spaces, in which geodesic rays are
uniformly strongly contracting, and CAT(0) spaces 
\cite{Sul14}.

\begin{proposition}
  If $X$ is a proper geodesic CA  metric space  then for all
  $\eta\in\bdry_cX$ the set $\mathcal{U}_\eta$ is a neighborhood basis at $\eta$ in
   $\bdry_c^{Gp}X$.
\end{proposition}
\begin{proof}
A standard topological argument shows that $\mathcal{U}_\eta$ is a
neighborhood basis at $\eta$ if and only if:
\begin{equation}
  \label{eq:1}
  \forall r>0,\, \exists R_\eta>r,\, \forall \zeta
  \in U(\eta,R_\eta),\, \exists R_\zeta>0 \text{ such that }
  U(\zeta,R_\zeta)\subset U(\eta,r)\tag{$\circledcirc$}
\end{equation}

 Suppose $\alpha$ is a contracting geodesic ray based at $o$. 
By the contraction alternative it is strongly contracting, so there
exists a $C\geq 0$ bounding its contraction function. 
For brevity, let us say that $\alpha$ is `$C$--strongly contracting'.
The Geodesic Image Theorem (GIT), \cite[cf. Theorem~7.1]{ArzCasGrub},
implies that if $\beta$ is a geodesic segment that stays at least
distance $2C$ from $\alpha$ then the diameter of $\pi_\alpha(\beta)$
is at most $4C$.
It follows easily that if $\beta$ is a geodesic ray
based at $o$ then $\alpha$ and $\beta$ are asymptotic if and only if
$\beta$ is contained in the closed $6C$--neighborhood of $\alpha$.
In fact, this can be improved to $5C$ by a further application of the
definition of strong contraction.

If $A$ is a contracting set and $B$ is bounded Hausdorff distance from
$A$ then $B$ is also contracting, with contraction function determined
by that of $A$ and the Hausdorff distance
\cite[Lemma~6.3]{ArzCasGrub}.
In particular, if $\alpha$ is  $C$--strongly contracting then there
exists a $C'$ depending only on $C$ such that every geodesic ray
$\alpha'$ based at $o$ and asymptotic to $\alpha$ is $C'$--strongly
contracting.
Thus, for a given $\eta\in\bdry_cX$ there exists a $C_\eta$ such that
every geodesic ray $\alpha\in\eta$ is $C_\eta$--strongly contracting.
 
\textbf{Claim:} Given $\eta\in\bdry_cX$ there exists
$K_\eta\geq 0$ such that for all $\zeta\in\bdry_cX\setminus \{\eta\}$
and all $\alpha\in\eta$, $\beta\in\zeta$, if $T(\alpha,\beta):=\max\{t\mid
d(\beta(t),\alpha)=2C_\eta\}$ then $|T(\alpha,\beta)-(\eta\cdot\zeta)_o|\leq K_\eta$.

Assuming the Claim, we show that condition (\ref{eq:1}) is satisfied.
Let $\eta\in\bdry_cX$ and  $r>0$.
Set $R_\eta:=r+2K_\eta+13C_\eta$.
For $\zeta\in U(\eta,R_\eta)$, set $R_\zeta:=(\zeta\cdot\eta)_o+K_\eta+K_\zeta+6C_\eta+4C_\zeta$.
Suppose that $\xi\in U(\zeta,R_\zeta)$.
Choose $\alpha\in\eta$, $\beta\in\zeta$, and $\gamma\in\xi$.
Let $x:=\gamma(T(\beta,\gamma))$ and let $y$ be a point of $\beta$ at distance $2C_\eta$ from
$x$. Let $z:=\beta(T(\alpha,\beta))$. Let $w:=\gamma(T(\alpha,\gamma))$.
\begin{align*}
  d(y,\alpha)&\geq d(y,z)-6C_\eta&\text{by the GIT}\\
&=d(o,y)-d(o,z)-6C_\eta\\
&\geq d(o,x)-d(o,z)-6C_\eta-2C_\zeta\\
&\geq
  (\xi\cdot\zeta)_o-(\zeta\cdot\eta)_o-K_\eta-K_\zeta-6C_\eta-2C_\zeta&\text{by the Claim, twice}\\
&\geq
  R_\zeta-(\zeta\cdot\eta)_o-K_\eta-K_\zeta-6C_\eta-2C_\zeta&\text{since
  }\xi\in U(\zeta,R_\zeta)\\
&=2C_\zeta=d(x,y)
\end{align*}
Since $d(x,y)\leq d(y,\alpha)$, 
the contraction property for $\alpha$ says the diameter of
$\pi_\alpha(x)\cup\pi_\alpha(y)$ is at most $C_\eta$.
With the GIT, this tells us the diameter of
$\pi_\alpha(\beta([T(\alpha,\beta),\infty)))\cup\pi_\alpha(\gamma([T(\alpha,\gamma),\infty)))$
is at most $9C_\eta$.
Thus, $d(w,z)\leq 13C_\eta$.
The Claim gives us $d(o,z)\geq R_\eta-K_\eta$, so $d(o,w)\geq
d(o,z)-13C_\eta\geq R_\eta-K_\eta-13C_\eta>
r+K_\eta$, which, by the Claim again, yields $(\xi\cdot\eta)_o\geq r$. 
Hence, $U(\zeta,R_\zeta)\subset U(\eta,r)$.

\medskip

It remains to prove the claim.
Let $\alpha,\,\alpha'\in\eta$ and $\beta,\,\beta'\in\zeta$ be arbitrary.

Consider $s,\,t\gg T(\alpha,\beta)$.
Let $\gamma$ be a geodesic from
$\alpha(s)$ to $\beta(t)$. Let $z$ be the last point on $\gamma$ at
distance $2C_\eta$ from $\alpha$.
Let $y\in\pi_\alpha(\beta(t))$. Let $x:=\beta(T(\alpha,\beta))$.
The GIT says the projection of the subsegment of $\beta$ between $x$
and $\beta(t)$ has diameter at most $4C_\eta$, as does the projection
of the subsegment of $\gamma$ from $z$ to $\beta(t)$.
Thus $d(x,y)\leq 6C_\eta$ and $d(y,z)\leq 6C_\eta$.
It follows that $|(\alpha(s)\cdot\beta(t))_o-d(o,y)|\leq 6C_\eta$, so:
\begin{equation}
  \label{eq:2}
 |(\alpha(s)\cdot\beta(t))_o-T(\alpha,\beta)|\leq 12C_\eta 
\end{equation}
 
Consider the effect of replacing $\alpha$ with $\alpha'$. 
For every $t$ we have
$d(\beta(t),\alpha)\geq t-T(\alpha,\beta)-6C_\eta$, and $\alpha$ and
$\alpha'$ have Hausdorff distance at most $5C_\eta$, so
$d(\beta(t),\alpha')\geq t-T(\alpha,\beta)-11C_\eta$.
Since $d(\beta(T(\alpha',\beta)),\alpha')=2C_\eta$ we have
$T(\alpha',\beta)\leq T(\alpha,\beta)+13C_\eta$. 
The argument is symmetric in $\alpha$ and $\alpha'$, so we conclude:
\begin{equation}
  \label{eq:3}
  |T(\alpha,\beta)-T(\alpha',\beta)|\leq 13C_\eta
\end{equation}

Now consider the effect of replacing $\beta$ with $\beta'$. 
The Hausdorff distance between them is at most $5C_\zeta$.
This does not admit any a priori bound in terms of $C_\eta$.
However, eventually points of $\beta$ are
closer to $\beta'$ than they are to $\alpha$, so we can invoke strong
contraction of $\alpha$ and the GIT, twice,  to say:
\[\diam \pi_\alpha(\beta([T(\alpha,\beta),\infty)))\cup
\pi_\alpha(\beta'([T(\alpha,\beta'),\infty)))\leq 9C_\eta\]
Which tells us:
\begin{equation}
  \label{eq:4}
  |T(\alpha,\beta)-T(\alpha,\beta')|\leq 13C_\eta
\end{equation}
Combining equations (\ref{eq:2}),  (\ref{eq:3}),  and (\ref{eq:4}), we
have, for all $s,\, s',\, t,\, t'$ sufficiently large, that
$|(\alpha(s)\cdot\beta(t))_o-(\alpha'(s')\cdot\beta'(t'))_o|\leq 50C_\eta$. 
Thus, for any $\alpha\in\eta$ and $\beta\in\zeta$ and for all
sufficiently large $s$ and $t$ we have
$|(\alpha(s)\cdot\beta(t))_o-(\eta\cdot\zeta)_o|\leq 50C_\eta$.
A further application of equation (\ref{eq:2}) completes the proof of
the Claim with $K_\eta:=62C_\eta$.
\end{proof}

\section{Pathological Examples}\label{sec:path}
Construct a proper geodesic metric space $X$ from rays $\alpha$, $\beta$,
and $\gamma_i$ for $i\in\mathbb{N}$ as follows. 
Identify $\alpha(0)$ and $\beta(0)$, and take this to be the basepoint $o$.
For each $i$ connect $\gamma_i(0)$ to $\alpha(i)$ and $\beta(i)$ by
segments of length $2^i$.
Then the $\gamma_i$ are strongly contracting, and $\alpha$ and $\beta$
are contracting rays whose contracting function $\rho$ can be taken to
be logarithmic.
The essential point is that the projection of $\gamma_i(0)$ to
$\alpha\cup\beta$ has diameter $2i$, while the distance from
$\gamma_i(0)$ to $\alpha\cup\beta$ is $2^i$.

The contracting boundary of $X$ consists of one point for each of the rays
$\alpha$, $\beta$, and $\gamma_i$, which we denote $\alpha(\infty)$,
$\beta(\infty)$, and $\gamma_i(\infty)$, respectively. 
Compute the Gromov products of boundary points:
$(\alpha(\infty)\cdot\gamma_i(\infty))_o=i=(\beta(\infty)\cdot\gamma_i(\infty))_o$,
while $(\alpha(\infty)\cdot\beta(\infty))_o=0$.
The sequence
$(\gamma_i(\infty))_i$ converges to both $\alpha(\infty)$ and
$\beta(\infty)$ in $\bdry_c^{Gp}X$.
In this example $\bdry^{Gp}_cX$ is compact but not Hausdorff.

Now consider the space $Y$ obtained from $X$ by
redefining, for each $i$, the length of the segment connecting
$\gamma_i(0)$ to $\beta$ to be $2^i-2i$.
The identity map is a
quasi-isometry, but in the new metric
$(\alpha(\infty)\cdot\gamma_i(\infty))_o=0$.
 The sequence $(\gamma_i(\infty))_i$ does not converge to
 $\alpha(\infty)$ in $\bdry_c^{Gp}Y$.
Thus, $\bdry_c\mathrm{Id}\from\bdry_c^{Gp}X\to\bdry_c^{Gp}Y$ is not continuous.
\medskip

Next, we construct a CAT(0) example.
Let $X'$ be the universal cover of the Euclidean plane minus a ball of
radius one. Parameterize $X'$ by polar coordinates $\mathbb{R}\times
[1,\infty)$.
Let $\alpha\from [0,\infty)\to X':t\mapsto (t,1)$ and $\beta\from
[0,\infty)\to X':t\mapsto (-t,1)$.
Each of these geodesic rays is $\pi$--strongly contracting. 

Let $X$ be the proper CAT(0) space obtained from $X'$ by
attaching, for each $i\in\mathbb{N}$, a geodesic ray $\gamma_i$ with
$\gamma_i(0)=(i,2^i)\in X'$.
These rays are also strongly contracting.

The contracting boundary of $X$ consists of points corresponding
to the $\gamma_i(\infty)$ and the two points $\alpha(\infty)$ and
$\beta(\infty)$.
Since $\alpha$ is strongly contracting, it follows that
$(\alpha(\infty)\cdot\gamma_i(\infty))_o=i$ up to bounded error.
Thus, the sequence $(\gamma_i(\infty))_i$ converges to
$\alpha(\infty)$ in $\bdry_c^{Gp}X$. 

Let $Y$ be the proper CAT(0) space obtained from $X'$ by
attaching, for each $i\in\mathbb{N}$, a geodesic ray $\gamma'_i$ with
$\gamma'_i(0)= (0,2^i)\in X'$.
Define $\phi$ to be the map $(t,r)\mapsto (t-\log_2(r),r)$ on $X'$, so
that $\phi(\gamma_i(0))=\gamma'_i(0)$.
This is a variation of the well-known logarithmic spiral
quasi-isometry of the Euclidean plane.
Extend $\phi$ to all of $X$ by  isometries $\gamma_i\to\gamma_i'$ for
each $i$. This gives a quasi-isometry $\phi\from X\to Y$, but points
in $\bdry_c^{Gp}Y$ are isolated, so $\bdry_c\phi\from\bdry_c^{Gp}X\to\bdry_c^{Gp}Y$ is not continuous.
\medskip

Interesting open questions remain: If $\phi\from X\to Y$
is a quasi-isometry between proper geodesic spaces that have cocompact isometry groups
and such that $X$ and $Y$ are CAT(0) (or, more generally, CA), must
$\bdry_c\phi\from \bdry_c^{Gp}X\to \bdry_c^{Gp}Y$ be a homeomorphism?
Must $\bdry_c^{Gp}X$ and $\bdry_c^{Gp}Y$ be homeomorphic?
Note that for visual boundaries of CAT(0) spaces the second question
is much harder than the first \cite{CroKle00}.



\end{document}